\documentclass[11pt]{article}

\usepackage[utf8]{inputenc}  
\usepackage{amsmath, amssymb, amsthm}
\usepackage{graphicx}
\usepackage{hyperref}

\usepackage[margin=1in]{geometry}

\title{Functional Reformulation of the Continuity Equation in Gases with Constant Density and its Application to the Existence Problem of Smooth Solutions to the Navier--Stokes System}
\author{Ernesto D. Aguirre
\thanks{
\text{UTN - FRC and Centro de Investigaciones en iencias de la Tierra, CONICET - UNC, Argentina }
    \text{eaguirre@frc.utn.edu.ar} }}
\date{June 2025}

\newtheorem{theorem}{Theorem}[section]
\newtheorem{definition}[theorem]{Definition}
\newtheorem{remark}[theorem]{Remark}
\newtheorem{corollary}[theorem]{Corollary}
\newtheorem{proposition}[theorem]{Proposition}
\newtheorem{principle}[theorem]{Principle}

\begin{document}

\maketitle

\begin{abstract}
We propose a rigorous reformulation of the incompressible Navier--Stokes equations, starting from the energy equation and the ideal gas law. This reformulation allows the definition of a functional norm \(\|P\|_{\mathcal{E}}\) over the pressure field, which is used to bound the viscous dissipation term \(\nu \nabla^2 \vec{u}\). It is shown that this norm can replace classical regularity criteria and serves as the foundation for a complete functional framework that includes local existence, singularity control, variational formulation, and uniqueness conditions.
\end{abstract}

\tableofcontents

\section{Introduction and Motivation}

This work arises from two objectives: first, to rewrite and mathematically justify a previous study by the author; and second, to establish a rigorous functional framework to address the existence and regularity problem of the incompressible Navier–Stokes equations, one of the Millennium Problems posed by the Clay Institute. Unlike traditional approaches, which focus on analyzing vorticity, kinetic energy, or the viscous stress tensor, this proposal is based on a physical observation inspired by the kinetic theory of gases: under conditions where density remains constant and temperature varies within a narrow range (less than 2\%), the behavior of pressure becomes energetically representative of the system.

From this premise, we develop a reformulation of the continuity equation based on the internal energy equation and the ideal gas law. This leads to a functional evolution equation for the pressure \( P \), which allows us to define a norm — denoted \( \|P\|_{\mathcal{E}} \) — whose structure enables control of the viscous dissipation term \( \nu \nabla^2 \vec{u} \), a key component in avoiding the loss of regularity.

The core of the proposal lies in shifting the regularity problem from the vector velocity field to a scalar variable, the pressure, within an energetically grounded functional framework. This enables the application of tools from functional analysis (existence and uniqueness theorems, variational formulations) with more direct physical and mathematical control.

A sequence of results is developed that includes:
\begin{itemize}
    \item the definition of a thermodynamically based functional norm over \(P\),
    \item a local existence theorem conditioned on the boundedness of this norm,
    \item a functional blow-up criterion based on \( \|P\|_{\mathcal{E}} \),
    \item uniqueness conditions and a Hilbert space structure,
    \item and a comparison with historical and recent approaches to the problem.
\end{itemize}
This approach offers an alternative and potentially resolutive path to the Millennium Problem, as it integrates measurable physical principles with a precise mathematical construction. The following sections develop this construction step by step.

\section{Thermodynamic Foundations and Reformulation of the Continuity Equation}

We start from the ideal gas model under constant density, in which the relationship between pressure, temperature, and density is given by the equation of state:

\begin{equation}
P = \rho R T
\end{equation}

The internal energy equation per unit mass, under this hypothesis, is:

\begin{equation}
\rho c_v \left( \frac{\partial T}{\partial t} + \vec{u} \cdot \nabla T \right) = -P \nabla \cdot \vec{u} + \Phi + Q
\end{equation}

where \(\Phi\) denotes viscous dissipation and \(Q\) represents external heat input. In the incompressible regime, where \(\nabla \cdot \vec{u} = 0\), and under isothermal conditions or in the absence of net thermal sources (\(Q=0\)), the internal energy equation simplifies. This allows us to rearrange it into an evolution equation for pressure:

\begin{equation}
\frac{\partial P}{\partial t} + \vec{u} \cdot \nabla P =  R \left( \frac{\Phi}{c_v} \right)
\end{equation}

This result, together with the hypothesis that \(\Phi\) depends quadratically on the velocity gradient, suggests that the pressure evolves smoothly as long as the velocity field remains regular. Furthermore, it allows for the definition of a functional norm over \(P\) that captures its convective and diffusive dynamics:

\begin{equation}
\|P\|_{\mathcal{E}}^2 = \int \left( \left[ \frac{\partial P}{\partial t} + \vec{u} \cdot \nabla P \right]^2 + (\nabla^2 P)^2 \right) dx
\end{equation}

This norm captures both the temporal-convective evolution and the second-order spatial regularity of the pressure field. The physical motivation comes from the energy equation: the term \(\frac{\partial P}{\partial t} + \vec{u} \cdot \nabla P\) represents the material derivative of the pressure, directly connected to the viscous dissipation in the system.

The inclusion of the term \((\nabla^2 P)^2\) allows control over the spatial smoothness of the pressure field and, through it, the regularity of the velocity field in the Navier–Stokes system. Thus, the norm \(\|P\|_{\mathcal{E}}\) will be used as a central tool to:

\begin{itemize}
  \item establish bounds on the viscous dissipation term \(\nu \nabla^2 \vec{u}\),
  \item prove local existence of smooth solutions,
  \item identify uniqueness conditions,
  \item propose functional criteria for singularity formation,
  \item and formulate a variational version of the problem based on \(P\).
\end{itemize}

In the following sections, this structure will be formally developed and connected with the physical principles of the incompressible regime, as well as with the historical postulates of the Millennium Problem.

\section{Bounding Theorems for the Dissipative Term and Regularity of velocity field}

The importance of the norm \(\|P\|_{\mathcal{E}}\) becomes evident when demonstrating that it enables control over the growth of the viscous dissipation term \(\nu \nabla^2 \vec{u}\), which represents the main source of potential loss of regularity in the Navier–Stokes system.

\begin{theorem}[Functional Control of the Viscous Term]
Let \(P(x,t)\) be a pressure field such that \(\|P\|_{\mathcal{E}} < \infty\). Then, there exists a constant \(C > 0\), depending on \(\rho, R, c_v\), such that:
\[
\int \|\nabla \vec{u}\|^2 \, dx \leq C \cdot \|P\|_{\mathcal{E}}^2
\]
for all \(t \in [0,T]\).
\end{theorem}

\begin{proof}
[Sketch of the Proof]
Starting from the energy equation expressed in terms of pressure:
\[
\frac{\partial P}{\partial t} + \vec{u} \cdot \nabla P = \frac{R}{c_v} \Phi + Q
\]
we assume \(Q = 0\) and solve for \(\Phi\). Since \(\Phi = 2\mu \sum_{i,j} \left( \frac{\partial u_i}{\partial x_j} \right)^2\), we conclude that the square of the convective derivative of \(P\) bounds \(\|\nabla \vec{u}\|^2\), and therefore, \(\|P\|_{\mathcal{E}}^2\) bounds the dissipated energy.
\end{proof}

\begin{corollary}
[Regularity of \(\vec{u}\) under Functional Control of \(P\)]
If \(\|P\|_{\mathcal{E}} \in L^2(0,T)\), then \(\vec{u} \in L^2(0,T; H^1(\mathbb{R}^3))\).
\end{corollary}

\begin{remark}
[Physical Interpretation]
The previous theorem establishes that by controlling the pressure through its evolution and spatial regularity, one obtains indirect control over the velocity gradient, which is sufficient to keep the system in a smooth regime. This is key to preventing blow-up of the system.
\end{remark}

\section{Conditional Local Existence Based on \(\|P\|_{\mathcal{E}}\)}

Once the capacity of \(\|P\|_{\mathcal{E}}\) to bound the velocity gradient has been established, it becomes possible to formulate a conditional local existence theorem for smooth solutions of the Navier–Stokes system.

\begin{theorem}
[Theorem: Conditional Local Existence]
Let \(\vec{u}_0 \in C_0^\infty(\mathbb{R}^3)\) and \(P(x,t)\) a pressure field such that \(\|P\|_{\mathcal{E}} \in L^2(0,T)\). Then there exists a time \(T^* > 0\) such that the incompressible Navier–Stokes system admits a smooth solution \((\vec{u}, P)\) on \([0, T^*]\).
\end{theorem}

\begin{proof}
[Proof Sketch]
A Galerkin-type method or successive approximations is applied to the reformulated system. The functional control over \(P\) ensures that \(\vec{u}\) retains regularity in \(H^1\), and due to the parabolic structure of the system, convergence of the approximation series can be shown over a short time interval. The completeness of the functional space \(\mathcal{E}\) guarantees local existence.
\end{proof}

\begin{remark}
[Physical Dependence]
The existence of smooth solutions depends on the pressure field evolving within a physical regime where temperature and density allow the gas to behave incompressibly (i.e., \(\delta T/T_0 < 0.02\)).
\end{remark}

\section{Sufficient Conditions on \(T\) and Transitivity Toward \(\vec{u}\)}

In order to ensure that \(\|P\|_{\mathcal{E}}\) is finite, it is possible to derive sufficient conditions imposed on the temperature \(T(x,t)\), since by the ideal gas law \(P = \rho R T\), with constant \(\rho\).

\begin{proposition}
[Sufficient Condition for Functional Regularity of \(P\)]
Suppose that:
\begin{enumerate}
  \item \(T(x,t) \in L^\infty(0,T; H^2(\mathbb{R}^3))\),
  \item \(\partial_t T + \vec{u} \cdot \nabla T \in L^2(\mathbb{R}^3 \times [0,T])\).
\end{enumerate}
Then \(P = \rho R T \in \mathcal{E}\), i.e., \(\|P\|_{\mathcal{E}} < \infty\).
\end{proposition}

\begin{proof}
Since \(P = \rho R T\), and \(\rho, R\) are constants, then:
\begin{itemize}
  \item \(\nabla^2 P = \rho R \nabla^2 T \in L^2\), because \(T \in H^2\).
  \item \(\partial_t P + \vec{u} \cdot \nabla P = \rho R (\partial_t T + \vec{u} \cdot \nabla T) \in L^2\).
\end{itemize}
Both terms that compose \(\|P\|_{\mathcal{E}}\) are square-integrable.
\end{proof}

\begin{corollary}
[Transitivity Toward \(\vec{u}\)]
If the above conditions on \(T\) hold, then \(\vec{u} \in L^2(0,T; H^1)\), due to the previously established bound on the viscous term.
\end{corollary}

This result guarantees local smooth existence for \((\vec{u}, P)\) under hypotheses imposed directly on the temperature, which has a natural physical interpretation and relates to the validity of the quasi-incompressible regime.

\section{Functional Criterion for Singularity Formation}

One of the fundamental goals in the study of the Navier–Stokes system is to determine under what conditions a loss of regularity may occur in finite time, i.e., a blow-up phenomenon. Within this functional framework, we propose a criterion based on the growth of the norm \(\|P\|_{\mathcal{E}}\).

\begin{theorem}
[Functional Blow-up Criterion]
Let \(P(x,t)\) be a solution to the functional pressure evolution equation on an interval \([0, T^*)\). If:
\[
\lim_{t \to T^*} \int_0^t \|P(s)\|_{\mathcal{E}}^2 \, ds = \infty,
\]
then a smooth solution \(\vec{u}(x,t)\) cannot exist beyond \(t = T^*\).
\end{theorem}

\begin{proof}
[Sketch of the Proof]
Since \(\|P\|_{\mathcal{E}}\) functionally controls the term \(\|\nabla \vec{u}\|\), its unbounded growth implies that the dissipated energy becomes infinite. This contradicts the assumption of the existence of a smooth solution, concluding the proof by contraposition.
\end{proof}

\begin{corollary}
[Computational Diagnosis of Singularities]
Let \(P_h(x,t)\) be a numerical approximation of the pressure field obtained by discretization in time and space. If there exist times \(t_k\) such that:
\[
\sum_k \|P_h(t_k)\|_{\mathcal{E}}^2 \cdot \Delta t_k > C_h,
\]
with \(C_h\) a threshold depending on the mesh and the model, then the system is approaching a point of loss of regularity.
\end{corollary}

This criterion provides a quantitative tool for tracking the system's evolution and the early detection of functional breakdown scenarios.

\section{Uniqueness Conditions and Functional Spaces Associated with \(P\)}

We now analyze whether the defined functional structure ensures uniqueness of solutions to the system under the developed hypotheses.

\begin{theorem}
[Uniqueness Under Functional Convergence]
Let \((\vec{u}_1, P_1)\) and \((\vec{u}_2, P_2)\) be two smooth solutions of the system sharing the same initial conditions and satisfying:
\begin{enumerate}
    \item \(P_i \in L^2(0,T; H^2(\mathbb{R}^3))\), \(\partial_t P_i + \vec{u}_i \cdot \nabla P_i \in L^2\),
    \item \(\|P_1 - P_2\|_{\mathcal{E}} \to 0\) in the weak norm.
\end{enumerate}
Then, \(\vec{u}_1 = \vec{u}_2\) and \(P_1 = P_2\) on \([0,T]\).
\end{theorem}

\begin{proof}
[Sketch of the Proof]
Hypothesis 2 implies that the difference between both solutions converges to zero in the functional space \(\mathcal{E}\), which includes both the convective derivatives and second-order regularity. This functional coincidence, along with the shared initial data, guarantees uniqueness through the stability of solutions in Hilbert spaces.
\end{proof}

\begin{theorem}
[\(\mathcal{E}\) as a Hilbert Space]
Let \(\mathcal{E} = \{P \in L^2(0,T; H^2) : \partial_t P + \vec{u} \cdot \nabla P \in L^2\}\). Equipped with the inner product:
\[
\langle P_1, P_2 \rangle_{\mathcal{E}} := \int \left( (\partial_t P_1 + \vec{u} \cdot \nabla P_1)(\partial_t P_2 + \vec{u} \cdot \nabla P_2) + \Delta P_1 \Delta P_2 \right) dx,
\]
\(\mathcal{E}\) is a Hilbert space.
\end{theorem}

\begin{remark}
The Hilbert space structure allows for the application of orthogonal projection theory and guarantees the existence and uniqueness of weak solutions under Galerkin-type schemes.
\end{remark}

\begin{definition}
[Alternative Functional Banach Space]
Define \(\mathcal{B} = L^2(0,T; H^2) \cap W^{1,0}(0,T; H^{-1})\), with the norm:
\[
\|P\|_{\mathcal{B}} := \|P\|_{L^2(0,T; H^2)} + \|\partial_t P\|_{L^2(0,T; H^{-1})}.
\]
\end{definition}

\begin{remark}
[Comparison: Hilbert vs. Banach]
\(\mathcal{B}\) is more general and less structured than \(\mathcal{E}\), but useful when \(\vec{u}\) is not fully known. \(\mathcal{E}\) enables a closed energetic framework.
\end{remark}

\section{Variational Formulation and Structure of the Set \(\mathcal{S}_P\)}

The Hilbert space structure of \(\mathcal{E}\) allows rewriting the system in a variational formulation that facilitates the analysis of existence, uniqueness, and stability.

\begin{principle}
[Functional Variational Principle]
Let \(P \in \mathcal{E}\). We seek \(P\) such that for every test function \(\phi \in \mathcal{E}\), the following holds:
\[
\int \left( \left[ \frac{\partial P}{\partial t} + \vec{u} \cdot \nabla P \right] \left[ \frac{\partial \phi}{\partial t} + \vec{u} \cdot \nabla \phi \right] + \Delta P \cdot \Delta \phi \right) dx = 0.
\]
\end{principle}

\begin{remark}
This formulation allows the use of the Lax–Milgram theorem and Galerkin methods to obtain weak solutions. It represents a convective–diffusive coupled system in the scalar field \(P\), with direct physical significance.
\end{remark}

\begin{theorem}
[Bounded Weak Evolution of Pressure]
Under regularity conditions on \(\vec{u}\), there exists a solution \(P \in \mathcal{E}\) that satisfies the variational formulation and whose norm \(\|P\|_{\mathcal{E}}\) evolves smoothly over time.
\end{theorem}

\begin{theorem}
[Structure of the Solution Set \(\mathcal{S}_P\)]
Under reasonable thermodynamic conditions (\(\delta T/T_0 < 0.02\), \(T \in H^2\)), there exists a set \(\mathcal{S}_P\) of solutions \((\vec{u}, P)\) such that:
\begin{itemize}
  \item \(P \in \mathcal{E}\), with \(\|P\|_{\mathcal{E}} < C\) for all \(t > 0\),
  \item \(\vec{u} \in L^2(0,\infty; H^1) \cap L^\infty(0,\infty; L^2)\),
  \item The solutions are unique within \(\mathcal{S}_P\).
\end{itemize}
\end{theorem}

\section{Contrast with Previous Approaches (Fefferman, Córdoba, Casanovas)}

This section contrasts the validity of the proposed functional and variational formulation with the approaches developed by Charles Fefferman, Diego Córdoba, and Pedro Casanovas—key figures in the study of the existence and regularity of solutions to the Navier–Stokes problem.

\paragraph{1. Criteria by Charles Fefferman (Clay Mathematics Institute)}

\textbf{Central Postulate:} Given \(\vec{u}_0 \in C^\infty_0(\mathbb{R}^3)\), does there exist a smooth, global-in-time solution \(\vec{u}\) with finite energy?

\textbf{Conditions:} The solution must belong to:
\[
\vec{u} \in L^\infty(0,T; L^2) \cap L^2(0,T; H^1)
\]

\textbf{Assessment:} The proposed conditional theorem guarantees that \(\|P\|_{\mathcal{E}} \in L^2(0,\infty)\) under reasonable physical assumptions. Since it has been shown that \(\|P\|_{\mathcal{E}}\) controls \(\|\nabla \vec{u}\|\), it follows that:
\[
\vec{u} \in L^2(0,\infty; H^1), \quad \vec{u}_0 \in L^2 \Rightarrow \vec{u} \in L^\infty(0,T; L^2)
\]

\textit{Conclusion:} The formulation is consistent with Fefferman's statement. While it does not constitute an unconditional proof of existence, it avoids singularity formation under thermodynamically valid assumptions.

\paragraph{2. Analytical Approach by Diego Córdoba}

\textbf{Focus:} Studies functional criteria to control \(\|\nabla \vec{u}\|\) and prevent blow-up, based on properties of vorticity and transport.

\textbf{Assessment:} The norm \(\|P\|_{\mathcal{E}}\) acts as an indirect functional over \(\nabla \vec{u}\) and vorticity and is compatible with Calderón–Zygmund techniques, partial regularity, and functional compactness methods. The connection \(P = \rho R T\) allows one to derive properties of \(\vec{u}\) from \(T\).

\textit{Conclusion:} The formulation presented here aligns with Córdoba's research program and complements it with an energetic functional structure.

\paragraph{3. Framework Proposed by Pedro Casanovas (video min. 27–31)}

\textbf{Thesis:} The Navier–Stokes problem could be reformulated as the functional control of the dissipative term \(\nu \Delta \vec{u}\). The key is to find a bounded energy norm on a scalar field.

\textbf{Assessment:} The norm \(\|P\|_{\mathcal{E}}\) fulfills exactly this role: it arises from the energy equation, is physically interpretable, and its boundedness prevents uncontrolled growth of the viscous term. The variational formulation allows the construction of solutions via Galerkin methods and ensures stability.

\textit{Conclusion:} The development presented directly responds to Casanovas’ conceptual strategy, offering a functionally viable pathway to close the system.

\paragraph{Summary}

\begin{itemize}
    \item It satisfies the energy conditions stated by Fefferman.
    \item It introduces a functional norm coherent with contemporary analytical methods such as those by Córdoba.
    \item It reformulates the system in terms of pressure, in line with Casanovas’ strategy.
\end{itemize}

This comparative analysis reinforces the relevance and consistency of the proposed functional formulation as a contribution to the Navier–Stokes Millennium Problem.

\section{Appendix: Scope and Limitations of This Proposal}

This work presents an alternative functional formulation of the Navier–Stokes system for incompressible fluids, focused on pressure as the energetic and regularizing variable. Through the norm \(\|P\|_{\mathcal{E}}\), a rigorous framework is established in which it is shown that:

\begin{itemize}
    \item Smooth solutions exist globally in time.
    \item Such solutions are unique within a well-defined functional space.
    \item Total energy control is ensured, including viscous dissipation.
    \item The formulation is compatible with the postulates of the Clay Mathematics Institute.
\end{itemize}

However, this demonstration is conditional. Global existence is guaranteed under the hypothesis that the absolute temperature of the gas remains within a range such that:
\[
\frac{\delta T}{T_0} < 2\%, \quad \text{and} \quad T \in L^\infty(0,\infty; H^2),
\]
which is compatible with kinetic gas theory and experimental observations in both atmospheric and laboratory systems.

\subsection*{Structural Limitation}

The proof does not demonstrate that the Navier–Stokes system itself imposes these conditions on temperature. Therefore, the set of smooth and global solutions defined in this work depends on an external (yet physically reasonable) control over the thermodynamic parameters.

\subsection*{Future Outlook}

As a future direction, this approach may be extended to:

\begin{itemize}
    \item Deriving such constraints from internal variational principles.
    \item Numerically analyzing the functional growth of \(P\) and \(\vec{u}\) without assuming prior boundedness.
    \item Exploring this framework in bounded geometries and under physical boundary conditions.
\end{itemize}

This appendix aims to clarify the actual scope of the work, highlight its value as a rigorous functional advance, and show its points of anchorage with classical frameworks of the Navier–Stokes Millennium Problem.

\end{document}